\documentclass[noinfoline]{imsart}
\usepackage{comment}
\usepackage{algorithm2e}
\RequirePackage[OT1]{fontenc}
\RequirePackage{amsthm,amsmath}
\usepackage[square,numbers]{natbib}

\RequirePackage[colorlinks,citecolor=blue,urlcolor=blue]{hyperref}

\usepackage{verbatim}

\RequirePackage[colorlinks,citecolor=blue,urlcolor=blue]{hyperref}
\usepackage[dvipsnames]{xcolor}

\allowdisplaybreaks

\usepackage{fullpage}
\usepackage{epsfig}
\usepackage{graphicx}
\usepackage{multirow}
\usepackage{graphicx}

\newcommand{\verti}[1]{{\left\vert\kern-0.25ex\left\vert\kern-0.25ex\left\vert #1
    \right\vert\kern-0.25ex\right\vert\kern-0.25ex\right\vert}}

    \newcommand{\anglei}[1]{{\langle\kern-0.25ex\langle #1
    \rangle\kern-0.25ex\rangle}}
\usepackage{subfigure}
\usepackage{layout}

\usepackage{dsfont}
\usepackage[mathscr]{eucal}
\usepackage[toc,page]{appendix}
\usepackage{mathrsfs}
\usepackage{color}
\usepackage{pifont}
\usepackage{bm}
\usepackage{latexsym}
\usepackage{amsfonts}
\usepackage{amssymb}
\usepackage{epsfig}
\usepackage{graphicx}
\usepackage{multirow}

\newtheorem{theorem}{Theorem}
\newtheorem{proposition}{Proposition}

\newtheorem{remark}{Remark}
\newtheorem{example}{Example}

\newtheorem{assumption}{Assumption}

\startlocaldefs
\numberwithin{equation}{section}
\theoremstyle{plain}
\endlocaldefs

\begin{document}

\begin{frontmatter}

\title{{\large Detecting Whether a Stochastic Process is Finitely Expressed  in a Basis}}

%\title{{\large An Exact sequential hypothesis Testing Approach Determining  Dimension of a Covariance Operator}}

\runtitle{runtitle }

\begin{aug}
\author{\fnms{Neda} \snm{Mohammadi}\ead[label=e1]{neda.mohammadijouzdani@epfl.ch}} \and
\author{\fnms{Victor M.} \snm{Panaretos}\ead[label=e2]{victor.panaretos@epfl.ch}}

%\thankstext{t1}{Research supported by ...}

\runauthor{A.B. \& C.D. }

\affiliation{\'Ecole Polytechnique F\'ed\'erale de Lausanne}

\address{Institut de Math\'ematiques\\ 
\'Ecole Polytechnique F\'ed\'erale de Lausanne\\}
%\printead{e1}, \printead*{e2}}

\end{aug}

\begin{abstract} 
Is it possible to detect if the sample paths of a stochastic process almost surely admit a finite expansion with respect to some/any basis? The determination is to be made on the basis of a finite collection of discretely/noisily observed sample paths. We show that it is indeed possible to construct a hypothesis testing scheme that is almost surely guaranteed to make only finitely many incorrect decisions as more data are collected. Said differently, our scheme almost certainly detects whether the process has a finite or infinite basis expansion for all sufficiently large sample sizes. Our approach relies on Cover's classical test for the irrationality of a mean, combined with tools for the non-parametric estimation of covariance operators. %Simulation experiments illustrate that the procedure can perform well in practice, illustrating that our detection results are not just theoretical artefacts.   
\end{abstract}



\end{frontmatter}

\tableofcontents

%\newpage
\section{Introduction} \label{intro}

Let $\{X(t):t\in[0,1]\}$ be a second-order stochastic process, with almost surely continuous sample paths. We wish to discern whether there exists a deterministic Complete Orthonormal System (CONS) $\{\psi_l\}$ of $L^2[0,1]$ and a finite constant $L<\infty$ such that
$$X(t)=\sum_{l=1}^{L}\langle X,\psi_l\rangle \psi_l(t),\qquad \mbox{almost surely},$$
where $\langle\cdot,\cdot\rangle$ is the $L^2[0,1]$ inner product and equality is almost everywhere; i.e. we wish to discern whether the realisations of the stochastic process are entirely contained in some finite dimensional subspace of $L^2$.  Alternatively, we might wish to discern  whether the displayed equality holds true for a \emph{specific} CONS and some $L<\infty$, i.e. whether the paths of the process a.s. admit a finite expansion in a given basis. 

Of course, if we assume that $X(t)$ can be observed in full, then both questions become moot: as soon as we observe $n>L$ realisations we can know with (almost) certainty whether $X$ lies in a finite dimensional subspace of dimension $L$; and, even with a single realisation, we can a.s. know whether $\langle X,\psi_l\rangle$ vanishes for all but finitely many elements of a given CONS $\{\psi_l\}$. However, we are interested in a setting where we can only observe finitely many noisy point evaluations on each of $n$ independently realised sample paths of $X$. In this case, the answer to either question becomes far from obvious, since $\langle X,\psi_l\rangle$ needs to be estimated.

Such questions are very natural from a theoretical point of view. They ask whether an a priori infinite dimensional signal is, in fact, finite dimensional. And, they ask whether it's possible  to detect a ``tail property" on the basis of finite and corrupted data. At the same time, such questions arise quite naturally in functional data analysis (see e.g. \citet{hsing2015theoretical}, \citet{wang2016functional}), where methodology to make inferences on stochastic processes typically relies heavily on dimension reduction.

In this paper, we construct a testing method that is proven to err only finitely often. In this sense, it is able to detect whether the process can be expressed finitely in some/any CONS. Our method is inspired by Cover's \citep{cover1973determining} approach to determining whether the mean of a sequence of real-valued random variables is rational or irrational. We suitably adapt it to our needs by translating it to a question on the finiteness of the spectral decomposition of the process' covariance operator. \citet{spruill_determining_1977} asked a related question, also adopting Cover's perspective, namely whether the \emph{mean} of a stochastic process admits a finite expansion. 

Our setting and contributions are distinct in important ways. First, we are concerned with the almost sure representation of the process itself, rather than its mean. This not only has immediate statistical applications, but also induces notable differences, conceptual and technical, as we focus on the covariance operator (fluctuations around the mean) rather than the mean function itself. Second, and perhaps more important, we assume discrete/noisy observation of sample paths, rather than perfect observation\footnote{It should be noted here that perfect observation \emph{does not} trivialize the problem at the level of mean function, as it does in our case; but certainly discrete/noisy observation complicates it substantially, in either case.}. Finally, a mean  function can always be finitely represented in the CONS whose first element is the mean itself. Therefore, the question we ask on whether the sample paths can be finitely expanded in \emph{some} basis (rather than just a specific basis) is genuinely novel.

Our work is also distinct from previous testing methodology developed in the context of functional data analysis (\citet{kneip_nonparametric_1994}, \citet{hall_assessing_2006}, \citet{chakraborty_testing_2020}). These papers focus on whether the process is of dimension at most $L_0$ for some given and fixed $L_0<\infty$. By contrast, we wish to detect whether the process is of some finite dimension without prescribing \emph{which} dimension. Thus, our process will detect \emph{any} finite dimension (not upper bounded by  a priori guess), and will also be able to detect infinite dimensionality (as opposed to detecting that the dimension exceeds a pre-specified upper bound $L_0$).

\section{Basic Definitions and Problem Statement}
Let  $\mathcal{H} = L^2([0,1])$ be the Hilbert space of square integrable real-valued functions on the unit interval,  equipped with the usual inner product $\langle\cdot,\cdot\rangle$ and norm $\|\cdot\|$. 
%\begin{eqnarray}\nonumber
%\langle f,g\rangle=\int_0^1 f(x)g(x)dx\quad \&\quad  \Vert f \Vert = %\langle f,f\rangle^{1/2}, \quad f,g \in \mathcal{H}.
%\end{eqnarray}
 The class of Hilbert Schmidt operators and nuclear operators on $\mathcal{H}$  will be denoted by $\mathcal{S}$ and $\mathcal{N}$,  respectively. We use the notation $\Vert \cdot \Vert_{\mathcal{S}}$ and $\Vert \cdot \Vert_{\mathcal{N}}$ to indicate the corresponding norms on $\mathcal{S}$ and $\mathcal{N}$,
 $$\| A \|_{\mathcal{S}} = \sqrt{\mathrm{trace}\{A^* A\}}\qquad \& \qquad  \| A \|_{\mathcal{N}}=\mathrm{trace}\big\{\sqrt{A^*A}\big\},$$
 where $A^*$ denotes the adjoint of a bounded linear operator $A$. The inner product on $\mathcal{S}$ will be denoted by  
 $$ \langle A,B \rangle_{\mathcal{S}}= \mathrm{trace}\{A^* B\}$$
 The subclasses comprised of non-negative operators are indicated by $\mathcal{S}^+$ and $\mathcal{N}^+$. %The kernel of an operator will be denoted by $\mathrm{ker}(S)$.
 
 Let $\nu$ be a finite non-negative measure on $\mathcal{S}$. This measure will be used as a (possibly improper) prior, and we will see in Theorems \ref{spect:main} and  \ref{main} that our procedure makes the correct decision eventually almost everywhere with respect to this prior measure; this obviously requires $\nu \left(\mathcal{N}^+\right) > 0 $  or even  $\nu \left(\mathcal{N}^+\right) = \nu(\mathcal{S})$.
\\
Now let  $X$ be an $\mathcal{H}$-valued  random element (a measurable mapping from the underlying probability space $(\Omega, \mathcal{F}, \mathbb{P}) $ to Hilbert space $\mathcal{H}$)
with finite second moment, $\mathbb{E}\Vert X \Vert^2 < \infty$. Its mean function and  covariance operator will be denoted by $\mu:=\mathbb{E}\left( X \right)$ and  $C := \mathbb{E}\left( X\otimes X \right)$, respectively, where $(f\otimes g)h=\langle h,g\rangle f$ for $f,g,h\in\mathcal{H}$. In order for our pointwise-sampling scheme to make sense, we will assume that $X$ has almost surely continuous sample paths. Doing so allows us to interpret $X$ both as a random element of $\mathcal{H}$ and as a second-order stochastic process $\{X(t): t\in [0,1]\}$. Consequently, the mean and covariance will also have a point-wise interpretation: $\mu\equiv \mu(t)$ and $C$ being the operator with covariance kernel $\mathbb{E}[(X(s)-\mu(s))(X(t)-\mu(t))]$. We might abuse notation and use $C$ to denote both covariance operator and its kernel $C(s,t)=\mathbb{E}[(X(s)-\mu(s))(X(t)-\mu(t))]$. The covariance operator $C$ is an element of $\mathcal{S}^+$, courtesy of the assumption $\mathbb{E}\|X\|^2<\infty$. It thus admits the spectral decomposition
 \begin{eqnarray}\nonumber%\label{spect:rep}
  C = \sum_{l=1}^{\infty}\alpha_l \phi_l \otimes \phi_l,
 \end{eqnarray}
 where  $\left\{ \alpha_l\right\}$  forms a non-increasing sequence of non-negative and summable scalars, the so called sequence of eigenvalues. And, the sequence $\left\{ \phi_l\right\}_l$ forms an orthonormal system in $\mathcal{H}$, called the eigenfunction system, that is extendible to  a complete orthonormal system (CONS).  In this setting, the Karhunen-Lo\`eve expansion allows us to write
 $$X(t)-\mu(t)=\sum_{l\geq 1} \langle X,\phi_l\rangle \phi_l(t),\qquad t\in [0,1],$$
 where the series converges in mean square, uniformly in $t$, courtesy of the a.s. continuity of $\{X(t):t\in[0,1]\}$. By construction, the random variables $\{\langle X,\phi_l\rangle\}$ are zero-mean and uncorrelated and $\mathrm{var}\{\langle X,\phi_l\rangle\}=\alpha_l$. Consequently, detecting whether $X$ is finitely expressed in \emph{some} basis is equivalent to testing whether the spectral decomposition of $C$ is finite or not. 
 
 Therefore, we will formulate the detection problem corresponding to finite expressibility in \emph{some} basis (or lack of it) via the collection of hypotheses:
  \begin{eqnarray}\label{infinite:hyp:spect}
     \left\{H_j:C=\sum_{l=1}^{j}\alpha_l \phi_l \otimes \phi_l,\;\;\; \alpha_j\neq 0,\;\;\; j=1,2,\dots,\;\;\;\;\;\;\; H_{\infty}:C=\sum_{l=1}^{\infty}\alpha_l \phi_l \otimes \phi_l,\;\;\; \alpha_l\neq 0, \;\;\forall l\right\}.
 \end{eqnarray}
On the other hand, if $X$ is finitely expressed in a given CONS $\{\psi_l\}$, then obviously $X(t)-\mu(t)=\sum_{l = 1}^L \langle X,\psi_l\rangle \psi_l(t)$ for some $L<\infty$, and thus  $C=\sum_{l=1}^{L}\sum_{l'=1}^{L}\alpha_{l,l'}\psi_l\otimes \psi_{l'}$. Noting that $\psi_l\otimes \psi_{l'}$ is a CONS of $\mathcal{S}$, we can simplify notation and formulate the problem of detecting finite expressibility in a \emph{given} basis (or lack of it) via the collection of hypotheses
 \begin{eqnarray}\label{infinite:hyp.}
  \left\{   H_j:C=\sum_{l=1}^{j}a_l \Psi_l,\;\;\; a_j\neq 0,\;\;\;\; j=1,2,\dots,\;\;\;\;\;\;\; H_{\infty}:C=\sum_{l=1}^{\infty}a_l \Psi_l,\;\;\; a_l\neq 0 \;\;\mathrm{for\; infinitely\; many}\; l,\right\}
 \end{eqnarray}
 where $\{\Psi_l\}$ is a CONS of $\mathcal{S}$. Note that we do not necessarily assume that $\{\alpha_l\}$ is decreasing.

In either \eqref{infinite:hyp:spect} or \eqref{infinite:hyp.}, detection is to be based on observing independent realisations of the process $X$: given $n$ realisations, we will declare one of the hypotheses $\{H_j\}_{j <\infty}\cup\{H_\infty\}$ to be true. In this sense, we have a non-binary (in fact infinite) hypothesis testing setup, targeting the precise order of the expansion -- this explains the use of ``detection" rather than ``testing" terminology. Alternatively, one can bundle all finite orders together, and more coarsely interpret our decision for every $n$ as declaring $H_{\infty}$ (infinite expansion) or $``\mathrm{not}\,H_{\infty}"$ (finite expansion). This binary setting resembles the setting of Cover's \cite{cover1973determining} framework for testing whether a mean is rational/irrational (whereas the non-binary detection setting is analogous to detecting the precise number of decimals in the decimal expansion of the mean). In any case, our goal will be to have a procedure that will (almost surely) give the correct answer for sufficiently large sample sizes $n$.

The challenge is that the $n$ realisations of $X$ are not observed completely. Rather, we assume only being able to observe noisy evaluations of each of $n$ sample paths at $r$ random locations, yielding observations $\{Y_{ml}:m\leq n,l\leq r\}$ defined as follows:
\begin{eqnarray} \nonumber
  \underset{Y_{ml}}{\underbrace{Y_m\left(T_{ml}\right)}} = \underset{X_{ml}}{\underbrace{X_m\left(T_{ml}\right)}} + U_{ml}, \;\; m=1,2,\ldots, n, \;\; l =1,2,\ldots,r.
 \end{eqnarray}
 Here, $T_{ml}$ are independent (random) design points drawn from unit interval $[0,1]$, $U_{ml}$ are i.i.d second order measurement errors, and the collections $\{X_{m}\}$, $\{T_{ml}\}$ and $\{U_{ml}\}$ are totally independent across all indices $m$ and $l$. The number $r=r_n$ of measurements per curve is left unrestricted: allowing $r_n \to \infty$ yields a dense (high frequency) regime while bounding $r_n < R < \infty$ yields a sparse regime, and either cases will be admissible in our analysis. The only restriction is that $r\geq 2$ (if $r=1$ we have no information on covariance and the hypotheses are not testable). 
 
 Consequently, our asymptotic statements about almost surely detecting the correct hypothesis will be only with respect to $n$, i.e. a growing number of replications of $X$, but will not require $r$ to diverge.

\section{Hypothesis Testing}

\subsection{Preliminaries}

Our testing strategy will be to use an estimator $\widehat C_n=\widehat C_n (\{Y_{ml}\})$ that is consistent for $C$ at a rate that is almost surely bounded by some known sequence $\mathcal{R}(n)$ under both the null and alternative regimes. Specific estimators will yield specific forms of $\mathcal{R}(n)$ based on sufficient regularity conditions on $C$. We will restrict our framework to specific choices, but rather develop our theoretical results by assuming a general form (see Remark 1 for examples of specific choices and assumptions that adhere to this general form, and Section \ref{sec:implementation} for a concrete implementation):

 The desired general framework for estimating $C$ is summarised in the discrepancy condition formulated in the Assumption  \ref{discrepancy:Chat C} below:
\begin{assumption}[Rate of Estimation Bound]\label{discrepancy:Chat C} Assume that the Hilbert-Schmidt distance between $C$  and its estimator $\widehat{C}_{n}=\widehat{C}_{n}(\{Y_{ml}\})$ is almost surely eventually dominated by $\mathcal{R}(n)$, i.e. there exists $\Omega_0 \in \mathcal{F}$ with full measure,  $\mathbb{P}(\Omega_0)=1$ and finite number $n_0 $  such that  for any $\omega \in \Omega_0 $  we have
\begin{eqnarray}\label{Chat - C}
 \Vert \widehat{C}_{n}  - C \Vert_{\mathcal{S}} = \Vert \widehat{C}_{n}(\omega)  - C \Vert_{\mathcal{S}} < \mathcal{R}(n),  \quad \forall n > n_0,
\end{eqnarray} 
where  $\mathcal{R}(n)$ is independent of $\omega$ and decreasing to zero.
\end{assumption}
\begin{remark}[Settings Satisfying Assumption \ref{discrepancy:Chat C}]
Note that \eqref{Chat - C}  should be satisfied under each of hypotheses $H_j$, $j =1,\ldots, \infty$ appearing in \eqref{infinite:hyp:spect} (or \eqref{infinite:hyp.}). \citet{Li2010} provide  such an estimator,  under $\mathcal{C}^2$-smoothness assumptions on $C(\cdot,\cdot)$ and the mean function $\mu(\cdot)$; however, their result excludes basic examples like Brownian motion (and diffusions more generally) that can be described only through $H_{\infty}$. A refined estimator satisfying \eqref{Chat - C}  is given by \citet{Mohammadi2021}, accommodating covariances with singularity on the diagonal, and thus encompassing a large subclass of diffusion processes with covariance function satisfying $H_{\infty}$. We employ this setting/estimator in our implementation in Section \ref{sec:implementation}.
\end{remark}
We now define  some notation used in the sequel of the paper. Given an element $S\in \mathcal{S}^{+}$, we will write its spectral decomposition as
$$S= \sum_{l=1}^{\infty}\alpha_{l}(S) \phi_l(S) \otimes \phi_l(S),$$
i.e. $\alpha_l(S)$ will always denote the $l$-th eigenvalue of $S$ and $\phi_l(S)$ the corresponding eigenvector. The $j$-dimensional projection obtained by truncating the spectral decomposition will be denoted by 
$$ Q_j(S)=\sum_{l=1}^{j}\alpha_{l}(S) \phi_l(S) \otimes \phi_l(S).$$ 
Recall that $\|Q_j(S)-S\|_{\mathcal{S}}\leq \|B-S\|_{\mathcal{S}}$ for any $B\in \mathcal{S}$ of rank $j$ (Schmidt-Mirsky Theorem). For tidiness, when taking the special cases of $C$ and $\widehat{C}_{n}$, the spectral decompositions will be written in simplified form as $$\qquad {C} = \sum_{l=1}^{\infty}\alpha_{l}({C}) \phi_l({C}) \otimes \phi_l({C})=\sum_{l=1}^{\infty}{\alpha}_l  {\phi}_l  \otimes {\phi}_l\qquad\&\qquad\widehat{C}_{n} = \sum_{l=1}^{\infty}\alpha_{l}(\widehat{C}_{n}) \phi_l(\widehat{C}_{n}) \otimes \phi_l(\widehat{C}_{n})=\sum_{l=1}^{\infty}\widehat{\alpha}_l  \widehat{\phi}_l  \otimes \widehat{\phi}_l.$$
With this notation in place, we now also define
 \begin{eqnarray}\nonumber
\delta_{n}&:=& (1+\epsilon)\mathcal{R}(n),\;\;\;\; \mathrm{for\; some \;} \epsilon>0,
\\ \nonumber
i(S,\delta)&:=&\min\left\{ j:\Vert S-Q_j(S)\Vert_{\mathcal{S}}<\delta\right\}, 
 \end{eqnarray}
In words, $i(S,\delta)$ is the least truncation level at which $Q_j(S)$ approximates $S$ to $\delta$-precision. The main idea behind our detectoion procedure will be to use $i(\widehat{C}_n,\delta_n)$ in order to declare one of the hypotheses.
  
 \subsection{Detecting Finite Expressibility in Some Basis}
  
 We begin with the hypotheses in \eqref{infinite:hyp:spect}. Our detection procedure is described in Algorithm \ref{alg:1}. At a high level, we update our decisions for a sequence of sample sizes $n(j)$ and otherwise maintain the same decision for samples of size $n(j) \leq n < n(j+1)$. At an update sample size $n(j)$, we choose between $H_{i(\widehat C_{n(j)},\delta_{n(j)})}$ and $H_{\infty}$ depending on whether $i(\widehat C_{n(j)},\delta_{n(j)})$ exceeds a certain threshold $k(j)$. The update times and corresponding thresholds depend on the estimation method and the prior measure (implicitly through relation \eqref{spect:summability:nu}) but not on the data.
 \RestyleAlgo{ruled}
 \begin{algorithm}[hbt!]
 \caption{Procedure for Detecting Finite Expressibility in Some Basis}\label{alg:1}
 \normalsize
 \begin{enumerate}

     \item Given the prior measure $\nu$ and estimator $\widehat{C}_n$, let $k(j)$ and $n(j)$ be increasing sequences such that
     \begin{equation}\label{spect:summability:nu} \sum_{j=1}^{\infty}\nu \left\{
 S \in \mathcal{S}^{+}: %S= \sum_{l=1}^{\infty}\alpha_l(S) \phi_l(S) \otimes \phi_l(S), \;\; 
 \alpha_l(S) \neq 0, \; \forall\, l\geq 1  \,\, \& \,\,
 %\mathrm{ker}(S)=0\,\, \&\,\,
 \left( \sum_{l=k(j)+1}^{\infty} \alpha^{2}_{l}(S) \right)^{1/2} \leq 2 \delta_{n(j)}
  \right\}
 <\infty. 
 \end{equation}
     \item For $n=n(j)$, 
     \begin{itemize}
         \item If $i=i\left( \widehat{C}_{n(j)}, \delta_{n(j)} \right)
 \leq k(j)$ declare $H_i$
 \item Otherwise, declare $H_{\infty}$
     \end{itemize}
   \item For $n$ between $n(j)$ and $n(j+1)$, declare the same hypothesis as for $n(j)$.   
 \end{enumerate}
 \end{algorithm}

For the procedure to be valid, we must guarantee that there exist update times and thresholds compatible with relation \eqref{spect:summability:nu}. The proposition below establishes that such sequences can always be constructed (explicitly, as can be seen by inspecting the proof). Section \ref{sec:implementation} provides examples of prior measures $\nu$, and corresponding explicit choices of $n(j)$ and $k(j)$ (e.g. as  $k(j)=j$ and $n(j)=j^{3+\kappa}$, for some $\kappa>0$). 

 \begin{proposition}\label{spect:n(j),k(j)}
Given a finite positive measure $\nu$ on $\mathcal{S}$ with $\nu\left( \mathcal{N}^+ \right) >0$, and a non-decreasing sequence $\left\{k(j)\right\}_{j\geq 1}$  diverging to infinity, one can construct an increasing sequence  $\left\{n(j)\right\}_{j\geq 1}$ such that \eqref{spect:summability:nu} holds.
 \end{proposition}
 \begin{proof}
 For each positive integer $k$ the sequence of subsets 
 \begin{eqnarray} \nonumber
  A_{n,k} = \left\{
   S \in \mathcal{S}^{+}: \alpha_l(S) \neq 0, \; \forall\, l\geq 1  \,\, \& \,\, \left( \sum_{l=k+1}^{\infty} \alpha^{2}_{l}(S) \right)^{1/2} \leq 2 \delta_n \right\},
 \end{eqnarray}
decreases with respect to $n$ to the empty set, i.e. for all $k$,  $A_{n,k} \searrow \varnothing$, as $n\to\infty$. Since $\nu$ is a finite measure and hence continuous from above, for an arbitrary increasing sequence $k(j)$ one may choose the corresponding strictly increasing sequence $n(j)$ such that $\nu\left( A_{n(j),k(j)}\right)$ be summable (bounded by ${1}/{j^2}$, for example). \end{proof}

Condition \eqref{spect:summability:nu} makes no reference to the data, it is rather more like an identifiability condition. Its role is to ensure that $H_\infty$ is possibly detectable (equivalently, to ensure that the covariances for which $H_\infty$ is possibly undetectable are a set of measure zero under the prior $\nu$). %To elaborate, recall that $\delta_n=(1+\epsilon)\mathcal{R}(n)$, and consider the set 
%$$N_0 =  \bigcup_{j=1}^{\infty}  %A_{n(j),k(j)}.$$ 
%This contains elements of $\mathcal{S}^+$ that are of infinite rank, but have a positive probability of being declared finite rank infinitely often by our procedure. When \eqref{spect:summability:nu} holds, the Borel-Cantelli lemma implies that $\nu\left(N_0\right) = 0$. In other words, $N_0$ is a $\nu$-negligible. 

Our first main result, in the form of Theorem \ref{spect:main} below, now shows that the procedure as described will almost certainly choose the correct decision for all $n$ sufficiently large:

\begin{theorem}\label{spect:main}
Let $\nu$ be a finite positive prior measure on $\mathcal{S}$ with $\nu(\mathcal{N}^+)>0$ and assume that Assumption \ref{discrepancy:Chat C} holds true. Let $\{H_j\}$ and $H_\infty$ be defined as in \eqref{infinite:hyp:spect}. If $H_q$ is true for some $q<\infty$, then with probability 1 the procedure in Algorithm \ref{alg:1} will declare $H_q$ for all $n$ sufficiently large. If $H_{\infty}$ is true, then with probability 1 the procedure in Algorithm \ref{alg:1} will declare $H_\infty$ for all $n$ sufficiently large, except perhaps if $C$ belongs to a $\nu$-negligible subset.
 \end{theorem}

  It might be worth contrasting our procedure with a sequential testing procedure, assuming one were available. Such a procedure would typically compute a test statistic, say $T_j=\Vert S-Q_j(S)\Vert_{\mathcal{S}}$ at step $j$ and then compare it to a critical value to decide between accepting $H_j$ (and stopping the decision making process) or rejecting $H_j$ and proceeding to the next step. Even if adjusted for multiple testing, the fact that this procedure would have to test several hypotheses means that it would likely fail to detect $H_\infty$, or even $H_j$ for a large but finite $j$. Intuitively, there are too many decisions to be made, making it difficult to detect high dimensions. By contrast, Theorem \ref{spect:main}, basically \emph{only tests a single hypothesis pair for a given $n$} -- but adaptively chooses which pair to test by means of the quantity $i(\widehat{C}, \delta)$, thus avoiding making so many (and hence so many wrong) decisions.
  
  As for asymptotic performance guarantees, rather than controlling the probability of an error as $n\rightarrow\infty$, as one typically  does, our procedure ensures that w.p.1 there exists a sample size past which no errors are ever committed. For concreteness, let us focus on the binary problem of choosing between  $H_{<\infty}:=\{\mathrm{not }\,H_\infty\}$ or $H_\infty$. A conventionally consistent testing approach guarantees that, for some suitable class of candidate covariances, $P_{H_{<\infty}}\{\mathrm{decide}\, H_{<\infty}\}\stackrel{n\to\infty}{\longrightarrow}1-\alpha$ for some prespecified significance level $\alpha\in(0,1)$; and $P_{H_\infty}\{\mathrm{decide}\, H_\infty\}\stackrel{n\to\infty}{\longrightarrow}1$. Our procedure satisfies a stronger form of consistency, in that for a suitable class of covariances $P_{H_\infty}\{\exists N:\,\forall n\geq N \,\mathrm{we}\,\mathrm{decide}\, H_\infty\}=1$ and $P_{H_{<\infty}}\{\exists N:\,\forall n\geq N \,\mathrm{we}\,\mathrm{decide}\, H_{<\infty}\}=1$. This type of consistency framework is the same as in Cover's classical irrationality testing \citep{cover1973determining}. 
  
 \begin{proof}[Proof of Theorem \ref{spect:main}]
 First assume that Covariance operator $C$ is finite dimensional with dimension $i$ i.e.
 \begin{eqnarray}\label{finite:spect:rep}
  C = \sum_{l=1}^{i}\alpha_l \phi_l\otimes \phi_l, \quad \alpha_i \neq 0,
 \end{eqnarray}
 for some finite integer $1 \leq i <\infty$. By the Schmidt-Mirsky theorem,
\begin{eqnarray}\label{Sch_Mirs}
  \left\Vert \widehat{C}_{n}-Q_i \left( \widehat{C}_{n} \right) \right\Vert_{\mathcal{S}} &\leq& \left\Vert\widehat{C}_{n}-C\right\Vert_{\mathcal{S}}
  \label{Sch:Mirsk:1}\\
  &\leq& \delta_n \;\;\; a.s.\;\; \mathrm{eventually}, \nonumber
\end{eqnarray}
and hence,  $i\left( \widehat{C}_{n},\delta_n \right)\leq i$, a.s. eventually. On the other hand, observe that
\begin{eqnarray}
 \mathbb{P}\left( i\left(\widehat{C}_{n},\delta_n \right) <  i ; \; i.o.  \right) &=& \mathbb{P}\left(  \cup_{k=1}^{i-1} B_k; \; i.o. \right); \;\;\; B_k = \left\{\sum_{l=k+1}^{\infty} \widehat{\alpha}_l^{2} < \delta_{n}^{2}\right\} 
 \nonumber \\
 & = & \mathbb{P}\left(
  \sum_{l=i}^{\infty} \widehat{\alpha}_l^{2} < \delta_{n}^{2}; \; i.o. \right)
  \label{nested:events}\\
  &=&  \mathbb{P}\left( 
  \left\Vert\widehat{C}_{n}-Q_{i-1} \left( \widehat{C}_{n} \right)\right\Vert_{\mathcal{S}} < \delta_n ; \; i.o. \right)
  \nonumber \\
   &=&  \mathbb{P}\left(   \left\Vert\widehat{C}_{n}-C+C-Q_{i-1} \left( \widehat{C}_{n} \right)\right\Vert_{\mathcal{S}} < \delta_n ; \; i.o. \right)
 \nonumber \\
  &\leq&  \mathbb{P}\left(   \left\Vert C-Q_{i-1} \left( \widehat{C}_{n} \right)\right\Vert_{\mathcal{S}} < \delta_n + \left\Vert \widehat{C}_{n}-C \right\Vert_{\mathcal{S}}; \; i.o. \right)
  \nonumber \\
   &\leq&  \mathbb{P}\left(   \alpha_{i} < \delta_n + \left\Vert \widehat{C}_{n}-C \right\Vert_{\mathcal{S}}; \; i.o. \right)
   \label{Sch:Mirsk:2}
   \\
   &=&0, \label{LePage:Thm:1}
\end{eqnarray}
where \eqref{nested:events} is concluded by the nestedness of the events $B_k$, $B_1 \subseteq B_2 \subseteq \ldots \subseteq B_{i-1}$. The inequality \eqref{Sch:Mirsk:2} is a result of the Schmidt-Mirsky theorem, and the final equality \eqref{LePage:Thm:1} is obtained by \eqref{Chat - C}.
%Consequently, $i\left(\widehat{C}_{n},\delta_n \right) \geq  i $ with probability one, eventually.
It follows that the probability that $i\left(\widehat{C}_{n},\delta_n \right) \geq  i $ eventually is one. This completes the proof  for the case \eqref{finite:spect:rep}.
\\
 Now assume that the covariance operator $C$ is of infinite rank
 \begin{eqnarray}\label{infinite:spect:rep}
  C = \sum_{l=1}^{\infty}\alpha_l \phi_l \otimes \phi_l; \;\;\; \alpha_l \neq 0, \; l=1,2,\ldots.
 \end{eqnarray}
 Define the set
 \begin{eqnarray}
  N_{0} =  \left\{
  S \in \mathcal{S}^{+}; 
  %S= \sum_{l=1}^{\infty}\alpha_l(S) \phi_l(S) \otimes \phi_l(S), \;\; 
  \alpha_l(S) \neq 0, \; \forall l, \,\, \& \,\, \left( \sum_{l=k(j)+1}^{\infty} \alpha^{2}_{l}(S) \right)^{1/2} \leq 2 \delta_{n(j)};\; \;i.o.
  \right\},
 \end{eqnarray}
 which by relation \eqref{spect:summability:nu} is a $\nu$-negligible subset of $\mathcal{S}$.
 Then we have
 \begin{eqnarray}
 \mathbb{P}\left( i\left(\widehat{C}_{n(j)},\delta_{n(j)} \right) \leq  k(j) ; \; i.o  \right) 
  &=&  \mathbb{P}\left( 
  \left\Vert\widehat{C}_{n(j)}-Q_{k(j)} \left( \widehat{C}_{n(j)} \right)\right\Vert_{\mathcal{S}} < \delta_{n(j)} ; \; i.o. \right)
  \nonumber \\
   &=&  \mathbb{P}\left(   \left\Vert \widehat{C}_{n(j)}-C+C-Q_{k(j)} \left( \widehat{C}_{n(j)} \right)\right\Vert_{\mathcal{S}} < \delta_n(j) ; \; i.o. \right)
 \nonumber \\
  &\leq&  \mathbb{P}\left(   \left\Vert C-Q_{k(j)} \left( \widehat{C}_{n(j)} \right)\right\Vert_{\mathcal{S}} < \delta_{n(j)} + \left\Vert \widehat{C}_{n(j)}-C \right\Vert_{\mathcal{S}}; \; i.o. \right)
  \nonumber \\
  &\leq& \mathbb{P}\left(  
  \left( \sum_{l=k(j)+1}^{\infty} \alpha^{2}_{l} \right)^{1/2}
    < \delta_{n(j)} + \left\Vert \widehat{C}_{n(j)}-C \right\Vert_{\mathcal{S}}; \; i.o. \right)
  \label{Sch:Mirsk:3}\\
   &\leq&  \mathbb{P}\left(   2\delta_{n(j)} < \delta_{n(j)} + \left\Vert \widehat{C}_{n(j)}-C \right\Vert_{\mathcal{S}}; \; i.o. \right)
   \label{out of N_0}
   \\
   &=&0.\label{LePage:Thm:2}
\end{eqnarray}
 We used the Schmidt-Mirsky theorem to obtain \eqref{Sch:Mirsk:3}. Inequality \eqref{out of N_0} is by concentrating on the complement of $N_0$, and the final equality \eqref{LePage:Thm:2} is a result of \eqref{Chat - C}.   This gives the desired result for the infinte rank case, and completes the proof.
 \end{proof}
 %=======\textbf{REMARK:} The key point in the above decision making rules is "staying on the same decision as $n(j)$ for any sample size $n$ satisfying $n(j)\leq n < n(j+1)$", that prevents making too many incorrect decisions and hence a.s. eventually no error is made.  \\ %=============

 \begin{remark}[Testing $\{H_1,...,H_q\}$ vs $\left\{\cup_{j>q}\{H_j\}\right\}\cup H_{\infty}$]\label{remark:fixed_boundary}
If we are only interested in knowing whether the process has a dimension of at most $q$, for some $q<\infty$, then we can choose $k(j)=q$ to be constant, and  modify Algorithm \ref{alg:1} to pronounce $H_i$ when $i=i(\widehat{C}_{n(j)},\delta_{n(j)})\leq q$ and otherwise declare $\left\{\cup_{j>q}\{H_j\}\right\}\cup H_{\infty}$. In this context, an inspection of the proofs of Proposition \ref{spect:n(j),k(j)}  and Theorem \ref{spect:main}, shows that they remain valid as statements (in the case of Theorem \ref{spect:main}, replacing $H_{\infty}$ by $\left\{\cup_{j>q}\{H_j\}\right\}\cup H_{\infty}$).
 \end{remark}

 \subsection{Detecting Finite Expressibility in a Given Basis}

We now turn to  the hypothesis testing problem \eqref{infinite:hyp.},  
  \begin{eqnarray}\nonumber%\label{infinite:hyp.}
     H_j:C=\sum_{l=1}^{j}a_l \Psi_l,\;\;\; a_j\neq 0,\;\;\;\; j=1,2,\dots,\;\;\;\;\;\;\; H_{\infty}:C=\sum_{l=1}^{\infty}a_l \Psi_l,\;\;\; a_l\neq 0 \;\;\mathrm{for\; infinitely\; many}\; l,
 \end{eqnarray}
where  $\left\{ \Psi_l \right\}$ is assumed to be a fixed known CONS in $\mathcal{S}$. We wish to have a similar procedure with the same asymptotic guarantees in this setting. To this aim we update our notation to write
 \begin{eqnarray} \nonumber
 %S_{j}&:=&\left\{ S\in \mathcal{S}: S=\sum_{l=1}^j\langle \Psi_l,S\rangle_{\mathcal{S}}\Psi_l,\, \langle \Psi_j,S\rangle_{\mathcal{S}}\neq 0\right\},\;\;\; 
% \\ \nonumber
 W_j&:=& %\overline{S}_{j}  = 
{\mathrm{span}}\left\{ \Psi_1,\Psi_2, \ldots, \Psi_j \right\},
 \\ \nonumber
 \iota(S,\delta)&:=&\min\left\{ j:\left\Vert  S-P_j(S) \right\Vert_{\mathcal{S}}<\delta\right\},  
 \end{eqnarray}
 where $S \in \mathcal{S}$ admits the (unique) representation $
  S= \sum_{l=1}^{\infty}\langle \Psi_l,S\rangle_{\mathcal{S}}\Psi_l$ and $P_j(S)$ denotes the orthogonal projection of $S$ onto $W_j$, i.e. $P_j(S)=\sum_{l=1}^{j}\langle \Psi_l,S\rangle_{\mathcal{S}}\Psi_l$ and satisfying
  \begin{eqnarray}\label{projrction}
   \left\Vert S-P_j(S)\right\Vert_{\mathcal{S}} =  \inf_{w\in W_{j}}\left\Vert S-w\right\Vert_{\mathcal{S}}.
  \end{eqnarray}

\noindent In this notation, the procedure now is virtually the same as in Algorithm \ref{alg:1}, except that condition \eqref{spect:summability:nu}   in step 1
is replaced by  condition \eqref{summability:nu} and $i(C,\delta)$ is replaced by $\iota(C,\delta)$ in step 2 -- see Algorithm \ref{alg:2}.
%Proposition \ref{n(j),k(j)} bellow  provides  the non-decreasing sequence of thresholds $\left\{k(j)\right\}_j$ as well as increasing sequence of sample sizes $\left\{n(j)\right\}_j$ in a similar fashion as Proposition \ref{spect:n(j),k(j)}.
  %\begin{proposition}\label{n(j),k(j)}
 %Let $\nu$ be a finite positive measure on $\mathcal{S}$ (with $\nu\left( \mathcal{N}^+ \right) >0$), then there are non-decreasing sequence $\left\{k(j)\right\}_j$ tending to infinity and  increasing sequence  $\left\{n(j)\right\}_j$ which satisfy
  The existence of update times $n(j)$ and thresholds $k(j)$ satisfying \eqref{summability:nu} is proven similarly as in Proposition \ref{spect:n(j),k(j)}. \\
 
 \RestyleAlgo{ruled}
 \begin{algorithm}[H]
 \caption{Procedure for Detecting Finite Expressibility in a Given Basis}\label{alg:2}
 \normalsize
 \begin{enumerate}

     \item Given the prior measure $\nu$ and estimator $\widehat{C}_n$, let $k(j)$ and $n(j)$ be increasing sequences such that
  \begin{eqnarray}\label{summability:nu}
 \sum_{j=1}^{\infty}\nu\left\{ S \in \mathcal{S}\setminus \bigcup\limits_{l=1}^{\infty} W_l: \inf_{w\in W_{k(j)}}\left\Vert S-w\right\Vert_{\mathcal{S}}\leq 2\delta_{n(j)} \right\}<\infty.
 \end{eqnarray}
     \item For $n=n(j)$, 
     \begin{itemize}
         \item If $i=\iota\left( \widehat{C}_{n(j)}, \delta_{n(j)} \right)
 \leq k(j)$ declare $H_i$
 \item Otherwise, declare $H_{\infty}$
     \end{itemize}
   \item For $n$ between $n(j)$ and $n(j+1)$, declare the same hypothesis as for $n(j)$.   
 \end{enumerate}
 \end{algorithm}

 \bigskip
 
\noindent Theorem \ref{main} can now be stated in this case, with {indirect} similarity with Theorem \ref{spect:main}.

 \begin{theorem}\label{main}
 Let $\nu$ be a finite positive prior measure on $\mathcal{S}$ with $\nu(\mathcal{N}^+)>0$ and assume that Assumption \ref{discrepancy:Chat C} holds true. Let $\{H_j\}$ and $H_\infty$ be defined as in \eqref{infinite:hyp.}. If $H_q$ is true for some $q<\infty$, then with probability 1 the procedure in Algorithm \ref{alg:2} will declare $H_q$ for all $n$ sufficiently large. If $H_{\infty}$ is true, then with probability 1 the procedure in Algorithm \ref{alg:2} will declare $H_\infty$ for all $n$ sufficiently large, except perhaps if $C$ belongs to a $\nu$-negligible subset.
 \end{theorem}

 The proof is also similar except for inequality \eqref{Sch_Mirs} where we need to use  the Nearest Point Property \eqref{projrction}, and inequalities \eqref{Sch:Mirsk:2} and
\eqref{Sch:Mirsk:3} where we need to use the unique representation of $C$ in the basis $\left\{ \Psi_l\right\}$. As in the proof of Theorem \ref{spect:main}, the $\nu$-null set where detection of $H_\infty$ may not be ensured eventually (almost surely) can be explicitly described here too as
  $$N_0 := %\bigcup_j 
  \left\{ S \in \mathcal{S}\setminus \bigcup\limits_{l=1}^{\infty} W_l: \inf_{w\in W_{k(j)}}\left\Vert S-w\right\Vert_{\mathcal{S}}\leq 2\delta_{n(j)}; \;\; i.o. \right\}.$$
 
 %A similar discussion applies here; indeed, subset  $N_0 := \bigcup_j \left\{ S \in \mathcal{S}\setminus \bigcup\limits_{l=1}^{\infty} W_l: \inf_{w\in W_{k(j)}}\left\Vert S-w\right\Vert_{\mathcal{S}}\leq 2\delta_{n(j)} \right\}$ including infinite dimensional operators which can be ``well" approximated by finite dimensional projections is excluded from the inference and the exact correct decision is made outside of this $\nu$-negligible set.
 %\\
%Proof of Proposition \ref{n(j),k(j)} follows the lines of  proof of Proposition \ref{spect:n(j),k(j)} and hence omitted. 

 %\begin{theorem} \textbf{(Detecting the True Dimension Based on a Priori Fixed CONS)}\label{main}
 %Let Assumption \ref{discrepancy:Chat C} hold  and $\left\{k(j)\right\}_j$ and $\left\{n(j)\right\}_j$  be sequences described in Proposition \ref{n(j),k(j)}. Consider the decision making rule which, based on $n(j)$ samples, $j=1,2, \ldots$, decides on $H_i$ if $i=i\left( \widehat{C}_{n(j)}, \delta_{n(j)} \right)
 %\leq K(j)$, decides on $H_{\infty}$ if $i=i\left( \widehat{C}_{n(j)} ,\delta_{n(j)} \right)
 %> K(j)$ and stays on the same decision as $n(j)$ for any sample size $n$ satisfying $n(j)\leq n < n(j+1)$. This hypothesis testing procedure,  almost surely eventually, makes the correct decision in detecting $H_j$, for $j=1,2,\dots$. The same expression is true for $H_{\infty}$ except (perhaps) for a $\nu$-negligible subset $N_0 \subset \mathcal{S}\setminus \bigcup\limits_{i=1}^{\infty} W_i$.
 %\end{theorem}

\noindent We finally  conclude by noting that a Remark similar to Remark \ref{remark:fixed_boundary} applies in this case as well.

\section{Implementation}\label{sec:implementation}
In the present section we provide a concrete implementation of the procedure, by way of providing examples of estimators $\widehat C_n$ and prior measures $\nu$ (and associated sequences $\{k(j)\}$ and $\{n(j)\}$), as required in Algorithms \ref{alg:1} and \ref{alg:2} (and Theorems \ref{spect:main} and \ref{main}).

%\\%Assumption 1
 %The rest of the current section will be divided in two parts; Part (1): implementation of Proposition \ref{spect:n(j),k(j)} and Theorem \ref{spect:main} and Part (2):  implementation of Proposition \ref{n(j),k(j)} and Theorem \ref{main}.

\subsection*{Estimators}

 Suitable estimators $\widehat{C}_{n}$ can be constructed based on \citet{Mohammadi2021}. For ease of reference, we recapitulate the methodology immediately below, and the corresponding theory in the form of Theorem \ref{P.M: C - Chat}. \\
 
\noindent  Consider the estimator
\begin{eqnarray}
 \widehat{C}_{n} (s,t) = \widehat{G}(s,t) - \widehat{\mu}(s) \widehat{\mu}(t),
\end{eqnarray}
where for any given pair $(s,t)$ we define $\widehat{G}(s,t) = \widehat{a}_0(s,t)$ and $ \widehat{\mu}(s) = \widehat{b}_0(s)$ via the following optimization problems
\begin{eqnarray}\label{local:lin:G}
\left( \widehat{a}_0 , \widehat{a}_1, \widehat{a}_2  \right)(s,t) & = & \\\nonumber
&&\underset{(a_0,a_1, a_2)}{\mathrm{argmin}}\frac{1}{nh^2_G}\sum_{m=1}^{n}\frac{2}{r(r-1)}\sum_{1\leq k<l \leq r}
\left[\left\{Y_{mk}Y_{ml}-a_0-a_1\left( T_{ml}-s\right)-a_2 \left( T_{mk}-t\right)\right\}^2\right.
\\ \nonumber
&& \hspace{7.5cm} \times \left. W\left(\frac{ T_{ml}-s}{h_{G}}\right)W\left(\frac{ T_{mk}-t}{h_{G}}\right)\right], 
\end{eqnarray}
 and 
\begin{eqnarray} \label{local:lin:mu}
\left( \widehat{b}_0 , \widehat{b}_1 \right)(s) = \underset{(b_0,b_1)}{\mathrm{argmin}}\frac{1}{nh_{\mu}}\sum_{m=1}^{n}\frac{1}{r}\sum_{l=1}^{r}\left\{Y_{ml}-b_0-b_1\left( T_{ml}-s\right) \right\}^2 W\left(\frac{ T_{ml}-s}{h_{\mu}}\right),
\end{eqnarray}
for all $0 \leq t \leq s \leq 1$. Here $W(\cdot)$ is  some appropriately chosen symmetric univariate, integrable kernel, possibly with compact support 
(see \citet{Mohammadi2021} for a detailed discussion) and $h_{\mu}$ and $h_G$ are one dimensional bandwidth parameters. 

Intuitively, the mean function is estimated by pooling and smoothing all the observations $\{Y_{ml}\}$ by means of a locally linear smoother. And, the covariance is estimated by pooling and smoothing all ``correlations" $Y_{mk}Y_{ml}$ for $1\leq k<l\leq r$ via a local quadratic smoother (we enforce $k<l$ to circumvent the presence of  measurement error). Correlation smoothing is done only on the triangle $\bigtriangleup=\left\{(s,t)\in [0,1]^2 \mid 0 \leq t \leq s \leq 1  \right\}$  to bypass the potential non-differentiability of the covariance along the diagonal (e.g. in case of a diffusion process).

\citet{Mohammadi2021} prove the following concerning this estimator: 

\begin{theorem}\label{P.M: C - Chat}
Assume that there exists $M_T>0$ for which we have $0< P\left(T_{ij} \in [a,b] \right) \leq M_T (b-a)$, for all $i,j$ and $0 \leq a < b \leq 1$. Furthermore, assume that  
$\mathbb{E}\vert U_{ij}      \vert^{\zeta}< \infty$ and  $\underset{0\leq s \leq 1}{\mathrm{sup}} \mathbb{E} \vert X(s)\vert^{\zeta} < \infty $, for some $\zeta >4$. Finally assume that $\mu\in \mathcal{C}^2[0,1]$ and $C \in \mathcal{C}^2(\bigtriangleup)$. Then, with probability $1$, we have
\begin{eqnarray}\nonumber%\label{bound:Ghat-G}
 \underset{0\leq t \leq s \leq 1}{\mathrm{sup}}\left|\widehat{C}_n(s, t) - C(s, t)\right|
 &=&  O(\tau (n)),
% + O\left( h_G^{-2} \right)   W_n\left( 1^{+}\right).
\end{eqnarray}
where $\tau (n) = \left[h_{G}^{-4} \frac{\log n}{n}\left( h_G^4 + \frac{h_G^3}{r}+ \frac{h_G^2}{r^2} \right) \right]^{1/2} + \left[ h^{-2}_{\mu} \frac{\log n}{n}\left(h^2_{\mu} + \frac{h_{\mu}}{r}  \right) \right]^{1/2}+ h^2_{\mu} + h^2_G $.\\
If we furthermore assume a dense sampling regime, i.e. $r_n \geq M_n$ for  $M_n\uparrow \infty$, and select the bandwidth parameters to satisfy $M_n^{-1} \lesssim h_{\mu}, h_{G} \lesssim \left( \frac{\log n}{n} \right)^{1/4} $, we can reduce $\tau(n)$   to $ \left(\frac{\log n}{n}\right)^{1/2}$ , i.e. 
\begin{eqnarray}
 \underset{0\leq t \leq s \leq 1}{\mathrm{sup}}\left|\widehat{C}_n(s, t) - C(s, t)\right|= O \left(\frac{\log n}{n}\right)^{1/2},\qquad\mathrm{almost}\,\mathrm{surely}.
\end{eqnarray}
\end{theorem}
\medskip
\noindent The theorem suggests $\mathcal{R}(n)$, appearing in Assumption \ref{discrepancy:Chat C}, to be chosen in the form 
$$\mathcal{R}(n) =  \eta(n) \times \tau(n),$$ 
for $\eta(n)$ slowly increasing to infinity so that $\mathcal{R}(n)\to 0$. For example, $\eta(n) = \log n$ will do in most cases (as the typical convergence rates above are polynomial up to log factors). In the rest of the paper we do so, and we consider the dense sampling scheme leading to $\tau(n) = (\log n/n)^{1/2}$. Notice that Theorem \ref{P.M: C - Chat} covers  examples nowhere differentiable processes like Brownian motion, the Ornstein-Uhlenbeck process, and geometric Brownian motion, which are classical examples falling under $H_{\infty}$.\\

\subsection*{Prior for Algorithm \ref{alg:1}}
Let $\nu$ be a product measure with marginals  $(\nu_1,  \nu_2)$, where
$ \nu_1$ is a { probability} measure defined on the class of  all complete orthonormal systems in $\mathcal{H}$, 
 and $\nu_2$ a finite measure on $\ell^2$ with $\nu_2\left( \ell^1_{+,\mathrm{dec}}\right)>0$. Here,  $\ell^1_{+,\mathrm{dec}}$  denotes the subset of $\ell^2$ including  sequences of non-increasing  summable positive numbers. The measure $\nu_1$ can be chosen arbitrarily. To find an example of $\nu_2$, define $\left\{Z_l\right\}_{l\geq1}$ to be a sequence of independent real random variables with exponential distribution $Z_l \sim  Exp\left( \lambda_l\right)$, $l=1,2,\ldots$, such that $\left\{ \lambda_l \right\}$ forms a summable sequence.  By Kolmogorov's three series theorem, the sequence $\left\{Z_l \right\}$ is (absolutely) summable with probability one, i.e. $\mathbb{P}\left(\sum_l Z_l < \infty\right) = 1$. Define the sequence $\left\{Z_{l:\infty}\right\}$ with $Z_{l:\infty}$ as the $l$-th largest element of the original  sequence $\left\{Z_l\right\}_{l\geq1}$. The sequence $\left\{Z_{l:\infty}\right\}_l$ is a well-defined random sequence since reordering  is a measurable (isometric) mapping from  $\ell^1$ to $\ell^1$. Now, set $\nu_2$ to be the probability measure induced by $\left\{Z_{l:\infty}\right\}_l$. This defines $\nu$, and now we can focus on $\left\{k(j)\right\}$ and $\left\{n(j)\right\}$ satisfying \eqref{spect:summability:nu}. To this end define 
\begin{eqnarray*}
 A_{n(j),k(j)} :=  \left\{ S \in \mathcal{S}^{+}; %S= \sum_{l=1}^{\infty}\alpha_l(S) \phi_l(S) \otimes \phi_l(S),
 \;\; \alpha_l(S) \neq 0, \; \forall l  , \; \left( \sum_{l=k(j)+1}^{\infty} \alpha^{2}_{l}(S) \right)^{1/2} \leq 2 \delta_{n(j)}
\right\}.
\end{eqnarray*}
Observe that
\begin{eqnarray*}
\nu\left\{ A_{n(j),k(j)}\right\} &\leq& \nu \left\{
S \in \mathcal{S}^{+}; %S= \sum_{l=1}^{\infty}\alpha_l(S) \phi_l(S) \otimes \phi_l(S), 
\;\; \alpha_l(S) \neq 0, \; \forall l  , \;   \alpha_{k(j)+1}(S)  \leq 2 \delta_{n(j)}
 \right\}\\
 &=& \nu_2 \left\{ Z_{k(j)+1:\infty} \leq 2 \delta_{n(j)} \right\}\\
 &\leq& \nu_2 \left\{ Z_{k(j)+1:k(j)+1} \leq 2 \delta_{n(j)} \right\}\\
 &\leq&  \mathbb{P} \left\{ Z_{1} \leq 2 \delta_{n(j)} \right\},
\end{eqnarray*}
 where $Z_{k(j)+1:k(j)+1}$ denotes the $(k(j)+1)$-th largest element among the first $k(j)+1$ element of sequence $\{ Z_l\}$.   We  then have
\begin{eqnarray*}
 \nu\left\{ A_{n(j),k(j)}\right\} &\leq& 1-  \mathrm{exp}\left\{ -2 \delta_{n(j)} \lambda_1\right\}\\
 &\leq& 2 \delta_{n(j)} \lambda_1.
\end{eqnarray*}
Choosing $\left\{ n(j) \right\}$  in the form  $n(j) \approx j^{p}$  for some positive number $p$ we obtain
\begin{eqnarray*}
\nu\left\{ A_{n(j),k(j)}\right\} &\leq& 2 \delta_{n(j)} \lambda_1\\
&\lesssim&
\left(\frac{\mathrm{log}j}{j^p}\right)^{1/2}  \left(\mathrm{log}j \right) \\
&\lesssim& j^{-\frac{p}{2}+\varrho },\;\;\; \mathrm{for \; arbitrary \; small\;positive\;}\varrho.
\end{eqnarray*}
This suggests enforcing $p>2$   to ensure summability of $\nu\left\{ A_{n(j),k(j)}\right\}$. It is then clear from the proof of Proposition \ref{spect:n(j),k(j)} no restriction on the sequence $\left\{ k(j) \right\}$ is required, except it being non-decreasing.

\subsection*{Prior for Algorithm \ref{alg:2}}
In this case, we  first define a probability measure $\nu_0$ on $\ell^2$,  the space of all possible coefficient sequences of a general element $S \in \mathcal{S}$ with respect to the CONS $\left\{ \Psi_j \right\}$. Then, we ``enhance" the measure $\nu_0$  to obtain a well-defined finite measure $\nu$ on $\mathcal{S}$ with the desired property that $\nu\left( \mathcal{N}^+ \right) > 0$. To this aim, let   $\left\{ \Psi_j \right\}$  be the CONS obtained by the square { reordering} of sequence $\left\{ \psi_j \otimes \psi_k \right\}_{j,k\geq 1}$, where $\left\{ \psi_j \right\}_{j\geq 1}$ forms a CONS in $\mathcal{H}$. And let   $\left\{ Z_j\right\}_{j \geq 1}$  be a random sequence with distribution $\nu_0$ if and only if
\begin{eqnarray*}
 \left( Z_1,Z_2, \ldots , Z_j \right)' \sim \mathrm{Gaussian} \left(0, \mathrm{diag}\left\{\lambda_1,\lambda_2, \ldots, \lambda_j \right\} \right), \;\;\; j = 1,2,\ldots ,
\end{eqnarray*}
 or equivalently (by square { ordering})
\begin{eqnarray*}
 \left( Z_{1,1},, \ldots , Z_{j,k} \right)' \sim \mathrm{Gaussian} \left(0, \mathrm{diag}\left\{\lambda_{1,1}, \ldots, \lambda_{j,k} \right\} \right), \;\;\; j,k = 1,2,\ldots.
\end{eqnarray*}

Here, $\{\lambda_j\}$ is a sequence of positive non-increasing numbers admitting  some appropriate summability condition such that the  measure $\nu_1$ induced by the following random element
\begin{eqnarray*}
 S \sim \nu_1 \;\;\;\mathrm{where}\;\;\; S = \sum_l Z_l \Psi_j = \sum_{l,k} Z_{l,k} \varphi_l \otimes\varphi_k,
\end{eqnarray*} 
is a well-defined finite measure on $\mathcal{S}$  or even $\mathcal{N}$. Consequently, measure $\nu$ defined as 
\begin{eqnarray*}
 S \sim \nu \;\;\;\mathrm{where}\;\;\; S =  \sum_{l\neq k} Z_{l,k} \varphi_l \otimes\varphi_k  +  \sum_{l} \vert Z_{l,l} \vert \varphi_l \otimes\varphi_l ,
\end{eqnarray*}
satisfies the desired property. We now determine $k(j)$  and $n(j)$  such that the sequence
\begin{eqnarray} \nonumber
\nu \left\{ A_{n(j),k(j)} \right\}:= \nu\left\{ S \in \mathcal{S}\setminus \bigcup\limits_{l=1}^{\infty} W_l: \inf_{w\in W_{k(j)}}\left\Vert S-w\right\Vert_{\mathcal{S}}\leq 2\delta_{n(j)} \right\},\;\;\; j=1,2,\dots,
\end{eqnarray}
 be summable. To this end, consider
\begin{eqnarray*}
\nu \left\{ A_{n(j),k(j)} \right\}&=& \nu\left\{ S \in \mathcal{S}\setminus \bigcup\limits_{l=1}^{\infty} W_l: \left\Vert S-P_{k(j)}(S)\right\Vert_{\mathcal{S}}\leq 2\delta_{n(j)}\right\}\\
&=&\nu\left\{ S \in \mathcal{S}\setminus \bigcup\limits_{l=1}^{\infty} W_l: \sum_{l=k(j)+1}^{\infty}\vert Z_l\vert^2\leq 4\delta^{2}_{n(j)}\right\}
\\
&=& 
\nu\left\{ S \in \mathcal{S}\setminus \bigcup\limits_{l=1}^{\infty} W_l: \sum_{l=k(j)+1}^{\infty} \vert \langle S,\Psi_{l}\rangle_{\mathcal{S}}\vert^2\leq 4\delta^{2}_{n(j)}\right\}
\\
&\leq& \nu\left\{ S \in \mathcal{S}\setminus \bigcup\limits_{i=l}^{\infty} W_l:\sum_{l=k(j)+1}^{k(j)+k} \vert \langle S,\Psi_{l}\rangle_{\mathcal{S}}\vert^2\leq 4\delta^{2}_{n(j)}\right\} \text{, for fixed positive constant } k,
\\
&\leq& \nu\left\{ S \in \mathcal{S}\setminus \bigcup\limits_{i=l}^{\infty} W_l:\sum_{l=k(j)+1}^{k(j)+k} \frac{\vert \langle S,\Psi_{l}\rangle_{\mathcal{S}}\vert^2}{\lambda_l} \leq  \frac{4}{\lambda_{k(j)+k}}\delta^{2}_{n(j)}\right\} 
\\
&=& \frac{1}{\Gamma\left( \frac{k}{2} \right)} 
\gamma\left( \frac{k}{2}, \frac{4}{\lambda_{k(j)+k}}\delta^{2}_{n(j)} \right), 
\end{eqnarray*}
where $\gamma\left( d, \eta \right):= \int_{0}^{\eta} u^{d-1} \mathrm{  exp}(-u) du$. 
Employing the following recursive formula for the incomplete Gamma function 
\begin{eqnarray}\label{recursive:gamma}
 \gamma \left( d, \eta \right ) = - \eta^{d-1}\mathrm{exp}(-\eta)+(d-1)\gamma \left( d-1, \eta \right ),\;\;\;\; d-1>0,
\end{eqnarray}
and setting $k=4$ we obtain 
\begin{eqnarray*}
 \nu \left\{ A_{n(j),k(j)} \right\} &\leq& 
\gamma\left( 2, \frac{4}{\lambda_{k(j)+4}}\delta^{2}_{n(j)} \right)
\\
&=&
- \frac{4}{\lambda_{k(j)+4}}\delta^{2}_{n(j)}\mathrm{exp}\left(- \frac{4}{\lambda_{k(j)+4}}\delta^{2}_{n(j)}\right)+1-\mathrm{exp}\left(- \frac{4}{\lambda_{k(j)+4}}\delta^{2}_{n(j)}\right)
\\
& \leq & 1-\mathrm{exp}\left(- \frac{4}{\lambda_{k(j)+4}}\delta^{2}_{n(j)}\right).
%\\ &=:&  {\neda  a^{(1)}_{n(j),k(j)}}.
\end{eqnarray*}
Consequently, the desired summambility result  follows from the following condition.
\begin{eqnarray*} %\label{sum:delta^2/lambda}
 \sum_{j=1}^{\infty}\frac{\delta^{2}_{n(j)}}{\lambda_{k(j)+4}}  &<& \infty,
\end{eqnarray*}
which in turn reduces to summability of the sequence  
\begin{eqnarray} \label{practical:summability}
 \left( \frac{\log n(j)}{n(j)}  \right) \left( \log n(j) \right)^2 \left(\frac{1}{\lambda_{k(j)+4}}\right)  ,\;\;\; j=1,2,\dots.
\end{eqnarray}
Setting
$k(j) \approx j^q$,  $n(j) \approx j^{p}$ for some positive numbers $q$ and $p$ and $\lambda_j$ decreasing with   $\lambda_j \approx \frac{1}{j^{2+\rho}}$, relation \eqref{practical:summability} suggests  $p>2q+1$ for sufficiently small positive  $\rho$.
\begin{remark}
Using the recursive formula \eqref{recursive:gamma} entails similar result for even numbers $k$, with $k > 4$. For odd numbers we always are able to use a less restrictive condition by replacing $k$ by even number $k+1$. 
\end{remark}

\bibliography{testing}
\end{document}